\documentclass[letterpaper, 10 pt, conference]{ieeeconf}  

\IEEEoverridecommandlockouts                              
\makeatletter
\newcounter{ALC@unique}
\makeatother

\let\labelindent\relax

\usepackage[draft=true, clevethm=true]{defaultpackages}

\usepackage{mynotation}
\usepackage{scenariotree}

\newcommand{\fa}{\tilde{f}}
\newcommand{\costidx}{0}
\newcommand{\confdyn}{C}
\newcommand{\RAOCP}{DR-OCP}
\newcommand{\nconst}{n_g}
\newcommand{\rv}{\xi}
\newcommand{\DR}[1]{\hat{#1}}
\newcommand{\fc}{\fa^{\DR{\law}_\hor}}

\newcommand{\nbeta}{n_{\beta}}
\newcommand{\augset}{\mathcal{Z}}

\newcommandx{\Vht}[2][2=\hor]{\DR{\cost}_{#2}^{(#1)}}  
\newcommandx{\Vtt}[2][2=\hor]{\tilde{\cost}_{#2}^{(#1)}} 
\newcommandx{\Vtil}[1][1=\hor]{\tilde{\cost}_{#1}}

\newif\ifDraft
\newif\ifExtraSteps
\newif\ifJournal
\newif\ifArxiv

\ExtraStepsfalse
\Journaltrue
\Drafttrue
\Arxivtrue

\newcommand{\extrastep}[1]{
\ifExtraSteps
{\color{ForestGreen!60!black} #1}
\else\fi
} 
\ExtraStepsfalse

\title{\LARGE \bf
Learning-Based Distributionally Robust Model Predictive Control of Markovian Switching Systems with Guaranteed Stability and Recursive Feasibility
}
\author{Mathijs~Schuurmans and Panagiotis~Patrinos
  \thanks{M. Schuurmans and P. Patrinos are with the Department 
  of Electrical Engineering (\textsc{esat-stadius}), KU Leuven, 
  Kasteelpark Arenberg 10, 3001 Leuven, Belgium.
  Email: \texttt{\{mathijs.schuurmans, panos.patrinos\}@esat.kuleuven.be}}
  \thanks{This work was supported by: FWO projects: No. G086318N; No. G086518N; 
  Fonds de la Recherche Scientifique -- FNRS, the Fonds Wetenschappelijk Onderzoek--Vlaanderen under EOS Project No. 30468160 (SeLMA), 
  Research Council KU Leuven C1 project No. C14/18/068 and the 
  Ford--KU Leuven Research Alliance project No. KUL0023.}
}

\begin{document}
\maketitle
\thispagestyle{empty}
\pagestyle{empty}

\begin{abstract}
We present a data-driven model predictive control scheme for chance-constrained Markovian switching systems with unknown switching probabilities.
Using samples of the underlying Markov chain, ambiguity sets of transition probabilities are estimated which include the true conditional probability distributions with high probability. These sets are updated online and used to formulate a time-varying, risk-averse optimal control problem. We prove recursive feasibility of the resulting MPC scheme and show that the original chance constraints remain satisfied at every time step. Furthermore, we show that under sufficient decrease of the confidence levels, the resulting MPC scheme renders the closed-loop system mean-square stable with respect to the true-but-unknown distributions, while remaining less conservative than a fully robust approach.\end{abstract}

\section{Introduction}

Stochasticity is inherent to processes emerging in a multitude of applications.
Nevertheless, we are often required to performantly control such systems, using only the limited information that is available. 
Over the past decades, the decreasing cost of data acquisition, transmission and storage has caused a surge in research interest in data-driven approaches towards control. More recently, as the focus in research is gradually shifting towards real-life, safety-critical applications, there has been an increasing concern for safety guarantees of such data-driven methods, which are valid in a finite-data regime (see \cite{hewing_learning-based_2020} for a recent survey).

In this paper, we focus on safe, learning-based control of Markovian switching systems. Control of this class of systems is widely studied and has been used to model systems stemming from a wide range of applications \cite{costa_discrete-time_2005,patrinos_StochasticModelPredictive_2014,
bernardini_ScenariobasedModelPredictive_2009}.

Many of these approaches, however, require knowledge of the transition kernel governing the switching behavior of these systems, although in practice, such information is typically unavailable. Although recently, data-driven methods have been proposed to address this problem  \cite{beirigo_online_2018,he_reinforcement_2019}, relatively little attention has gone to providing \emph{a priori} guarantees on stability and constraint satisfaction. 

Leveraging the framework of risk-averse \ac{MPC} \cite{sopasakis_RiskaverseModelPredictive_2019,sopasakis_risk-averse_2019c, singh_framework_2018} and its connection to \ac{DR} optimization \cite{rahimian_distributionally_2019}, we propose a learning-based \ac{DR} \ac{MPC} scheme, which is provably stable, recursively feasible and uses data gathered during operation to improve performance, rendering it less conservative than more traditional robust approaches \cite{ben-tal_RobustOptimization_2009}. Despite its growing popularity
\cite{yang_WassersteinDistributionallyRobust_2018,mohajerinesfahani_DatadrivenDistributionallyRobust_2018,vanparys_DistributionallyRobustControl_2015b, 
ugurlu_RobustOptimalControl_2018,singh_framework_2018
}, 
this approach has remained relatively unexplored for chance-constrained stochastic optimal control problems with ambiguity in the estimation of conditional distributions, in particular for Markovian switching dynamics.

We summarize our contributions as follows.
\begin{inlinelist} 
   \item We propose a general data-driven, \ac{DR} \ac{MPC} scheme for Markov switching systems with unknown transition probabilities, which is compatible with the recently developed framework of risk-averse \ac{MPC} \cite{sopasakis_RiskaverseModelPredictive_2019,sopasakis_risk-averse_2019c}. The resulting closed-loop system satisfies the (chance) constraints of the original stochastic problem and allows for online improvement of performance based on observed data.
   \item We state the problem in terms of an augmented state vector of constant dimension, which summarizes the available information at every time. This
   \ifJournal idea, which is closely related to that of \emph{sufficient statistics} \cite[Ch. 5]{bertsekas_DynamicProgrammingOptimal_2005} and \emph{belief states} in partially-observed Markov decision processes \cite{krishnamurthy_PartiallyObservedMarkov_2016}
   \fi allows us to formulate the otherwise time-varying optimal control problem as a dynamic programming recursion. 
   \item We provide sufficient conditions for recursive feasibility
   and mean-square stability of the proposed \ac{MPC} scheme, with respect to the true-but-unknown probability distributions.
\end{inlinelist}   
\subsection{Notation}
Let $\N$ denote the set of natural numbers and $\N_{>0} \dfn \N \setminus 0$. For two naturals $a,b \in \N, a \leq b$, we denote $\natseq{a}{b} \dfn \{ n \in \N \mid a \leq n \leq b \}$ and similarly, we introduce the shorthand $\seq{\md}{a}{b} \dfn (\md_{t})_{t=a}^{b}$ to denote a sequence of variables. We denote the extended real line by $\barre \dfn \Re \cup \{\pm \infty\}$ and the set of \emph{nonnegative} (extended) real numbers by $\Re_+$ (and $\barre_+$). The cardinality of a (finite) set $\W$ is denoted by $|\W|$. We write $f: X \rightrightarrows Y$ to denote that $f$ 
is a \emph{set-valued mapping} from $X$ to $Y$. 
Given a matrix $P \in \Re^{n\times m}$, we denote its $(i,j)$'th element by $\elem{P}{i}{j}$ and its $i$'th row as $\row{P}{i} \in \Re^m$. The $i$'th element of a vector $x$ is denoted $\idx{x}{i}$ whenever confusion with time indices is possible. 
$\vect(M)$ denotes the vertical concatenation of the columns of a matrix $M$. We denote the vector in $\Re^k$ with all elements one as $\1_k \dfn (1)_{i=1}^{k}$ and the probability simplex of dimension $k$ as $\simplex_k \dfn \{ p \in \Re_+^{k} \mid \trans{p} \1_k = 1\}$. We define the indicator function as $1_{x=y} = 1$ if $x=y$ and $0$ otherwise. 
Similarly, the characteristic function $\delta_{\Xfeas}: \Re^{n} \rightarrow \barre$ of a set $\Xfeas \subseteq \Re^{n}$ is defined by $\delta_{\Xfeas}(x) = 0$ if $x \in \Xfeas$ and $\infty$ otherwise.

\section{Problem statement}
We consider discrete time Markovian switching systems with dynamics of the form
\begin{equation} \label{eq:system-dynamics}
    x_{t+1} = f(x_t, u_t, \md_{t+1}), 
\end{equation}
where $x_t \in \Re^{\ns}, u_t \in \Re^{\na}$ are the state and control action at time $t$, respectively,
and $\md_{t+1}: \Omega \rightarrow \W$ is a random variable drawn from a discrete-time, time-homogeneous Markov chain $\w \dfn (\md_{t})_{t \in \N}$  defined on a probability space $(\Omega, \F, \prob)$ and taking values on $\W \dfn \natseq{1}{\nModes}$. We refer to $ \md_t$ as the \emph{mode} of the Markov chain at time $t$. The \emph{transition kernel} governing the Markov chain is denoted by $\transmat = (\elem{\transmat}{i}{j})_{i,j \in \W}$, where $\elem{\transmat}{i}{j} = \prob[\md_{t} = j \mid  \md_{t-1} = i]$.
\ifJournal As such, the probability space can be constructed by taking the sample space $\Omega = \W^{\infty}$, $\F = 2^\Omega$ and for any $( \md_t)_{t\in\N} \in \Omega$, $\prob[( \md_t)_{t\in\N}]= p_0 \prod_{t=0}^{\infty}\elem{\transmat}{ \md_t}{\md_{t+1}}$, where $p_0 \in \simplex_{\nModes}$ is the initial distribution. For simplicity, we will assume that the initial mode is known to be $\md$, so $p_0=(1_{i=\md})_{i\in \W}$.
\fi
We assume that the state $x_t$ and mode $ \md_t$ are observable at time $t$.
For a given state-mode pair $(x, \md) \in \Re^{\ns} \times \W$, we constrain the control action $u$ to the set $\Ufeas(x, \md)$, defined
\ifJournal
through $\nconst$ chance constraints of the form
\fi
\begin{multline} \label{eq:chance-constraint}
\Ufeas(x, \md) \dfn \{ u \in \Re^{\na}:\\
    \prob[g_{i}(x,u,\md,\mdnxt) > 0 \mid x,\md] \leq \alpha_i, \forall i \in \natseq{1}{\nconst}\},
\end{multline}
where $\mdnxt \sim \row{\transmat}{\md}$ is randomly drawn from the 
Markov chain $\w$ in mode $w$, and  
$g_i: \Re^{\ns} \times \Re^{\na} \times \W^2 \rightarrow \Re$ are constraint functions with corresponding constraint violation rates $\alpha_i$. By appropriate choices of $\alpha_i$ and $g_i$, constraint \eqref{eq:chance-constraint} can be used to encode robust constraints ($\alpha_i = 0$) or chance constraints ($0<\alpha_i<1$) on the state, the control action, or both. 
Note that the formulation \eqref{eq:chance-constraint} additionally covers chance constraints on the successor state $f(x,u,\mdnxt)$ under input $u$, conditioned on the current values $x$ and $\md$.

Ideally, our goal is to synthesize -- by means of a stochastic \ac{MPC} scheme -- a stabilizing control law $\law_{\hor}: \Re^{\ns} \times \W \rightarrow \Re^{\na}$, such that for the closed loop 
system \(x_{t+1} = f(x_t, \law_{\hor}(x_t,  \md_t), \md_{t+1})\), 
it holds almost surely (\as) that $\law_{\hor}(x_t, \md_t) \in \Ufeas(x_t, \md_t)$, for all $t \in \N$. Consider a sequence of $\hor$ control laws $\pol = (\pol_{k})_{k=0}^{\hor-1}$, referred to as a \emph{policy} of length $\hor$.
Given a stage cost $\ell:\Re^{\ns} \times \Re^{\na} \times \W \rightarrow \Re_{+}$, and a terminal cost $\Vf: \Re^{\ns} \times \W \rightarrow \Re_{+}$ and corresponding terminal set $\Xf$: $\bar{\Vf}(x,w) \dfn \Vf(x, \md) + \delta_{\Xf}(x)$, we can assign to each such policy $\pol$, a cost
\begin{equation} \label{eq:cost-function-stochastic}
    \cost_{\hor}^{\pol}(x, \md) \dfn \textstyle\E\big[\sum_{k=0}^{\hor-1} \ell(x_k, u_k,\md_k) + \bar{\Vf}(x_\hor,\md_\hor) \big], 
\end{equation}
where $x_{k+1} = f(x_k, u_k, \md_{k+1})$, $u_k = \pol_k(x_k, \md_k)$ and $(x_0, \md_0) = (x,\md)$, for $k \in \natseq{0}{\hor-1}$.
This defines the following stochastic \ac{OCP}. 

\begin{definition}[Stochastic \ac{OCP}]
For a given state-mode pair $(x, \md)$, the optimal cost of the stochastic \ac{OCP} is 
\begin{subequations} \label{eq:stochastic-OCP}
\begin{align}
    \cost_{\hor}(x, \md) = \min_{\pol} \cost_{\hor}^{\pol}(x, \md)
\end{align}
subject to 
\begin{align} \label{eq:stochastic-constraints}
    x_0&=x, \md_0=\md, \pol = (\pol_{k})_{k=0}^{\hor-1}, \\
    x_{k+1} &= f(x_k, \pol_{k}(x_k,\md_k), \md_{k+1}),\\
    \pol_{k}(x_k, \md_k) &\in \Ufeas(x_k ,\md_k), 
    \; \forall k \in \natseq{0}{\hor-1}.
\end{align}
\end{subequations}
We denote by $\Pi_{\hor}(x, \md)$ the corresponding set of minimizers.
\end{definition}

Let $(\pol^{\star}_{k}(x, \md))_{k=0}^{\hor-1} \in \Pi_{\hor}(x, \md)$, so that the stochastic \ac{MPC} control law is given by $\law_{\hor}(x, \md) = \pol^{\star}_0(x, \md)$. 
Sufficient conditions on the terminal cost $\bar{\Vf}$ and its effective domain $\dom \bar{\Vf} = \Xf$ to ensure mean-square stability of the closed-loop system, have been studied for a similar problem set-up in \cite{patrinos_StochasticModelPredictive_2014}, among others.

Both designing and computing such a stochastic \ac{MPC} law requires knowledge of the probability distribution governing the state dynamics \eqref{eq:system-dynamics}, or equivalently, of the transition kernel $\transmat$. 
In reality, $\transmat$ is typically not known but rather estimated from a finitely-sized sequence $\seq{\md}{0}{t}$ of observed values. Therefore, they are subject to some level of misestimation, commonly referred to as \emph{ambiguity}. 
The goal of the proposed \ac{MPC} scheme is to model this  
ambiguity explicitly, in order to be resilient against it, without 
disregarding available statistical information.
To do so, we introduce the notion of a \textit{learner state}, 
which is very similar in spirit to the concept of a \emph{belief state}, commonly used in control of partially observed Markov decision processes \cite{krishnamurthy_PartiallyObservedMarkov_2016}. It can be regarded as an internal state of the controller that stores all the information required to build a set of possible conditional distributions over the next state, given the observed data. We formalize this in the following 
assumption.

\begin{assumption}[Learning system] \label{assum:learner}
Given a sequence $\seq{\md}{0}{t}$ sampled from the Markov chain $\w$, we can compute \begin{inlinelist} \item a statistic $\lrn_t: \W^{t+1} \rightarrow \lrnset \subseteq \Re^{\nl}$, accompanied by a vector of confidence parameters $\beta_t  \in \I \dfn [0,1]^{\nbeta}$, which admit recursive update rules $\lrn_{t+1} = \learner(\lrn_t, \md_t, \md_{t+1})$ and $\beta_{t+1} = \confdyn(\beta_t)$, $t \in \N$; and \item an \emph{ambiguity set} $\amb: \lrnset \times \W \times [0,1] \rightrightarrows \simplex_{\nModes}: (\lrn, \md, \beta) \mapsto \amb_{\beta}(\lrn, \md)$\end{inlinelist}, mapping $\lrn_t$, $ \md_t$ and an element $\idx{\beta_{t}}{i}$ to a convex subset of the $\nModes$-dimensional probability simplex $\simplex_\nModes$ such that for all $t \in \N$,
    \begin{equation} \label{eq:high-confidence}
        \prob[\row{\transmat}{\md_{t}} \in \amb_{\idx{\beta_{t}}{i}}(\lrn_t,  \md_t)] \geq 1 - \idx{\beta_{t}}{i}.
    \end{equation} 
    We will refer to $\lrn_t$ and $\beta_t$ as the state of the learner and the confidence vector at time $t$, respectively.
\end{assumption}
\begin{remark}[confidence levels]
    Two points of clarification are in order. 
    First, we consider a vector of confidence levels, rather 
    than a single value. This is motivated by the fact that one would often wish to assign separate confidence levels to ambiguity sets 
    corresponding to the $\nconst$ individual chance constraints as well as the cost function of the data-driven \ac{OCP} (defined in \Cref{def:raocp} below). Accordingly, we will assume that $n_{\beta} = \nconst + 1$. 
    
    Second, the confidence levels are completely exogenous to the system dynamics and can in principle be chosen to be any time-varying sequence satisfying the technical conditions discussed further (see \Cref{prop:conditions-confidences} and \Cref{assum:confidences}).
    The requirement that the sequence $(\beta_{t})_{t\in\N}$ can be written as the trajectory of a time-invariant dynamical system
    serves to facilitate theoretical analysis of the proposed 
    scheme through dynamic programming. Meanwhile, it covers
    a large class of sequences one may reasonably choose, as illustrated in \Cref{ex:example-ambiguity-choice}.
\end{remark}
We will additionally invoke the following assumption on the confidence levels when appropriate.
\begin{assumption} \label{assum:confidences}
    The confidence dynamics $\beta_{t+1} = \confdyn(\beta_t)$ is chosen such that 
    \(
        \sum_{t=0}^{\infty} \idx{\beta_{t}}{i} < \infty, \forall i \in \natseq{1}{\nbeta}.
    \)    
\end{assumption}
In other words, we will assume that the probability of obtaining an
ambiguity set that contains the true conditional distribution (expressed by \eqref{eq:high-confidence}) increases sufficiently fast. This 
assumption will be of crucial importance in showing stability (see \Cref{sec:stability}).

To fix ideas, consider the following example of a learning system satisfying the requirements of \Cref{assum:learner}. 

\begin{example}[Transition counts and $\ell_1$-ambiguity] \label{ex:example-ambiguity-choice}
    A natural choice for the learner state is to take  $\lrn_t = \vect(\nsample(t))$,
    where $\nsample(t) = (\elem{\nsample}{\md}{\mdnxt}(t))_{\md,\mdnxt \in \W} \in \N^{\nModes \times \nModes}$ contains the mode transition counts at time $t$. That is, 
    $\elem{\nsample}{\md}{\mdnxt}(t) = |\{\tau \in \natseq{1}{t} : \md_{\tau-1}=\md, \md_{\tau} = \mdnxt\}|$, for all $\md,\mdnxt \in \W$. 
    \ifJournal
    Thus, since, 
    \[ 
        \begin{aligned}
            \elem{\nsample}{\md}{\mdnxt}(t+1) = \begin{cases}
                \elem{\nsample}{\md}{\mdnxt}(t) +1 &\text{if } \md_t = \md, \md_{t+1}=\mdnxt\\
                \elem{\nsample}{\md}{\mdnxt}(t) &\text{otherwise,}
            \end{cases}
        \end{aligned}
    \]
    it 
    \fi
    is clear that we can indeed write $s_{t+1} = \learner(s_t,  \md_t, \md_{t+1})$. Furthermore, following \cite{schuurmans_SafeLearningBasedControl_2019},
    we can uniquely obtain ambiguity sets parametrized as
    \[
        \ambTV_{\idx{\beta_{t}}{i}}(s_t, \md_t) \dfn \{p \in \simplex_\nModes : \|p - \hat{p}_{ \md_t} \|_1 \leq r_{ \md_t}(s_t, \idx{\beta_{t}}{i}) \}, 
    \]
    for $i \in \natseq{1}{\nbeta}$, 
    where $\hat{p}_{ \md_t}$ is the empirical estimate of the $ \md_t$'th row of the transition kernel (initialized to the uniform distribution if no transitions originating in mode $ \md_t$ have been observed) and the radii $r_{ \md_t}(s_t, \idx{\beta_{t}}{i})$ are chosen to satisfy \eqref{eq:high-confidence} by means of basic concentration inequalities, yielding a closed-loop expression $\bigo(\sqrt{- \row{\nsample}{\md_t}(t)^{-1} \ln(\beta_t)}),$ with $\nsample_{\md_t} = \tsum_{\mdnxt \in \W} \elem{\nsample}{\md_t}{\mdnxt}$.
    \ifJournal Both of these quantities depend solely on the learner state $s_t$, the confidence level $\beta_t$ and the 
    current mode $ \md_t$. \fi
    Finally, the sequence of the confidence levels $\beta_t$ remains to be selected. One particular family of sequences satisfying
    the additional requirement of \Cref{assum:confidences} is
    \ifJournal $\beta_t = b e^{-\tau t}$, with some arbitrarily chosen parameters $0 \le b \le 1$ and $\tau > 0$. Alternatively, if a slower decay rate is preferred, one may choose a sequence of the form 
    \(
        \beta_t = b (1+t)^{-q}, 
    \) with parameters $0 \leq b \leq 1$, $q > 1$. The former can trivially be identified with a stable linear system; the latter can be described by the recursion 
    \(
        \beta_{t+1} = b \beta_t {(\beta_t^{\nicefrac{1}{q}} + b^{\nicefrac{1}{q}})^{-q}},\, \beta_0 = b,
    \)
    both satisfying the requirements of \Cref{assum:learner}.
    \fi 
\end{example}

Equipped with a generic learning system of this form,
our aim is to find a data-driven approximation to the stochastic \ac{OCP} defined by \eqref{eq:stochastic-OCP}, which asymptotically attains the optimal cost while preserving stability and constraint satisfaction during closed-loop operation.
\section{Data-driven model predictive control} \label{sec:data-driven-MPC}

Given a learning system satisfying \Cref{assum:learner}, we define the augmented state $z_t = (x_t, \lrn_t, \beta_t, \md_t) \in \augset \dfn \Re^{\ns} \times \lrnset \times \I \times \W$, which evolves over time according to the dynamics
\begin{equation} \label{eq:augmented-dynamics}
    z_{t+1} = \fa(z_t, u_t, \md_{t+1}) \dfn \smallmat{f(x_t, u_t, \md_{t+1})\\ \learner(\lrn_t, \md_t, \md_{t+1})\\\confdyn(\beta_t)\\ \md_{t+1}},
\end{equation}
where $\md_{t+1} \sim \row{\transmat}{\md_t}$, for $t \in \N$.
Consequently, our scheme will result in a feedback law $\law: \augset \rightarrow \Re^{\na}$. To this end, 
we will formulate a \ac{DR} counterpart to the stochastic \ac{OCP} \eqref{eq:stochastic-OCP}, in which the expectation operator in the cost and the conditional probabilities in the constraint will be replaced by operators that account for ambiguity in the involved distributions. 

\subsection{Ambiguity and risk}
In order to reformulate the cost function \eqref{eq:cost-function-stochastic}, 
we first introduce an ambiguous conditional expectation operator, leading to a formulation akin to the Markovian risk measures utilized in \cite{sopasakis_RiskaverseModelPredictive_2019,ruszczynski_RiskaverseDynamicProgramming_2010b}. 
Consider a function $\rv: \augset \times \W \rightarrow \barre$, defining a stochastic process $(\rv_t)_{t \in \N} = (\rv(z_t, \md_{t+1}))_{t\in\N}$ on $(\Omega, \F, \prob)$, and suppose that 
the augmented state $z_t = z = (x, \lrn, \beta, \md)$ is given. For simplicity, let us assume for the moment that $\beta \in [0,1]$ is scalar; the following definition can be repeated for each component in the general case. The ambiguous conditional expectation of $\rv(z, \mdnxt)$, given $z$ is
\begin{equation} \label{eq:dr-expectation}
    \begin{aligned}
       \lrisk{\lrn}{\md}{\beta}[\rv(z,\mdnxt)] &\dfn \max_{p \in \amb_{\beta}(\lrn, \md)} \E_{p}[\rv(z,\mdnxt) | z] \\
        &= \max_{p \in \amb_{\beta}(\lrn, \md)} \tsum_{\mdnxt \in \W} p_\mdnxt \rv(z,\mdnxt).
    \end{aligned}
\end{equation}
Trivially, it holds that if the $w$'th row of the transition matrix lies in the corresponding ambiguity set: $\row{\transmat}{\md} \in \amb_{\beta}(\lrn,\md)$, then
\begin{equation*} 
    \begin{aligned}
    \lrisk{\lrn}{\md}{\beta}[\rv(z,\mdnxt)] &\geq \E_{\row{\transmat}{\md}} [\rv(z,\mdnxt)|z] 
    = \textstyle \sum_{\mdnxt \in \W} \elem{\transmat}{\md}{\mdnxt} \rv(z,\mdnxt).
    \end{aligned}
\end{equation*}
Note that the function $\lrisk{\lrn}{\md}{\beta}$ defines a coherent risk measure \cite[Sec. 6.3]{shapiro2009lectures}. We say that $\lrisk{\lrn}{\md}{\beta}$ is the risk measure \emph{induced by} the ambiguity set $\amb_{\beta}(\lrn,\md)$. 

A similar construction can be carried out for the chance constraints \eqref{eq:chance-constraint}. 
In their standard form, chance constraints lead to nonconvex, nonsmooth constraints. For this reason, they are commonly approximated using risk measures \cite{nemirovski2012safe}.
Particularly, the (conditional) \emph{average value-at-risk} (at level $\DR{\alpha} \in (0, 1]$ and with reference distribution $p \in \simplex_\nModes$) of $\rv$ is the coherent risk measure 
\begin{equation} 
    \begin{aligned}
        \label{eq:def-AVAR}
            &\AVAR_{\DR{\alpha}}^{p} [\rv(z,\mdnxt) \mid z]  
    \\ 
        =&\begin{cases}
            \min\limits_{t\in\Re} 
                t+\nicefrac{1}{\DR{\alpha}}\E_{p}\plus{\rv(z,\mdnxt)-t \mid z},&\DR{\alpha}\neq 0\\
            \max_{\mdnxt \in \W}[\rv(z,\mdnxt)],
                                & \DR{\alpha}=0.
        \end{cases}
    \end{aligned}
\end{equation} 
It can be shown (e.g., \cite[sec. 6.2.4]{shapiro2009lectures}) that if $p = \row{\transmat}{\md}$, then the following implication holds tightly
\begin{equation} \label{eq:risk-constraint-implication} 
    \begin{aligned}
    &\AVAR^p_{\DR{\alpha}}[\rv(z,\mdnxt) {\mid} z] {\leq} 0
    \Rightarrow \prob[\rv(z,\mdnxt) \leq 0 {\mid} z] \geq 1-\DR{\alpha}.
    \end{aligned}
\end{equation}
By exploiting the dual risk representation \cite[Thm 6.5]{shapiro2009lectures}, the left-hand inequality in \eqref{eq:risk-constraint-implication} can be formulated in terms of only linear constraints \cite{sopasakis_risk-averse_2019c}. As such, it can be used as a tractable surrogate for the original chance constraints, given perfect probabilistic information.
Accounting also for the ambiguity in the knowledge of $\row{\transmat}{\md}$ through $\amb_{\beta}(\lrn, \md)$, we define 
\begin{equation} \label{eq:ambiguous-chance-constraint}
    \lriskc{\lrn}{\md}{\beta}{\DR{\alpha}}[\rv(z,\mdnxt)] \dfn \hspace{-6pt} \max_{p \in \amb_{\beta}(\lrn,\md)} \hspace{-6pt} \AVAR^p_{\DR{\alpha}}[\rv(z,\mdnxt) \mid z] \leq 0.
\end{equation}
The function $\lriskc{\lrn}{\md}{\beta}{\DR{\alpha}}$ in turn defines a coherent risk measure.

We now present a condition on the choice of $\DR{\alpha}$ under which a constraint of the form \eqref{eq:ambiguous-chance-constraint} can be used as a tractable and safe approximation of  
a chance constraint when there is ambiguity in the probability distribution.

\begin{proposition} \label{prop:conditions-confidences}
    Let $\beta, \alpha \in [0,1]$, be given values with $\beta < \alpha$.
    Consider the random variable $\lrn: \Omega \rightarrow \lrnset$, denoting an (a priori unknown) learner state satisfying $\Cref{assum:learner}$,
    i.e., $\prob[\row{\transmat}{\md} \in \amb_{\beta}(s,\md)] \geq 1 - \beta$.
    If the parameter $\DR{\alpha}$ is chosen to satisfy 
    \(
        0 \leq \hat{\alpha} \leq \frac{\alpha - \beta}{1 - \beta} \leq 1, 
    \)
    then, for an arbitrary function $g : \augset \times \W \rightarrow \barre$, the following implication holds:
    \begin{equation} \label{eq:implication-prop3-1}
        \begin{aligned}
        \lriskc{\lrn}{\md}{\beta}{\hat{\alpha}}[g(z,\mdnxt)] \leq 0, \, \as
        \Rightarrow
        \prob[g(z,\mdnxt) \leq 0 {\mid} x,\md] {\geq} 1- \alpha.
        \end{aligned}
    \end{equation}
\end{proposition}

\begin{proof}

    If \(
        \lriskc{\lrn}{\md}{\beta}{\hat{\alpha}}[g(z,\mdnxt) ] \leq 0 
    \), \as, then by \eqref{eq:risk-constraint-implication}-\eqref{eq:ambiguous-chance-constraint},
    \[
        \prob[g(z,\mdnxt) \leq 0 \mid x,\md,\row{\transmat}{\md} \in \amb_{\beta}(\lrn, \md)] \geq 1 - \DR{\alpha}, \as  
    \] Therefore,
    \[
        \begin{aligned}
            &\prob[g(z,\mdnxt) \leq 0 \mid x,\md] \\
            &\geq 
            \prob[g(z,\mdnxt) \leq 0 \mid x,\md,\row{\transmat}{\md} \in \amb_{\beta}(\lrn, \md)]
             \prob[\row{\transmat}{\md} \in \amb_{\beta}(\lrn, \md)] \\ 
            &\geq (1-\DR{\alpha})(1-\beta).
        \end{aligned}
    \]
    Requiring that $(1-\DR{\alpha})(1-\beta) \geq (1 - \alpha)$ then immediately yields the sought condition.
\end{proof}
    Notice that the implication \eqref{eq:implication-prop3-1} in \Cref{prop:conditions-confidences} provides an \textit{a priori} guarantee,
    since the learner state is considered to be random. 
    In other words, the statement is made before the data is revealed. 
    Indeed, for a \emph{given} learner state $\lrn$ and mode $\md$, the ambiguity set $\amb_{\beta}(\lrn,\md)$ is fixed and therefore, the outcome of the event $E = \{\row{\transmat}{\md} \in \amb_{\beta}(\lrn,\md)\}$ is determined. Whether \eqref{eq:implication-prop3-1} then holds for these fixed values depends on the outcome of $E$.
    This is naturally reflected through the above condition on $\DR{\alpha}$ which implies that $\DR{\alpha} \leq \alpha$, and thus tightens the chance constraints that are imposed for a \emph{fixed} $\lrn$. Hence, the possibility that for this particular $\lrn$, the ambiguity set may not include the conditional distribution, is accounted for. This effect can be mitigated by decreasing $\beta$, at the cost of a larger ambiguity set. A more detailed study of this trade-off is left for future work.

\subsection{Risk-averse optimal control}

We are now ready to describe the \ac{DR} counterpart to the OCP \eqref{eq:stochastic-OCP}, which, when solved in receding horizon fashion, yields the proposed data-driven MPC scheme. 

For a given augmented state $z = (x, \lrn, \beta, \md) \in \augset$, we use \eqref{eq:ambiguous-chance-constraint} to define the \ac{DR} set of feasible inputs $\DR{\Ufeas}(z)$ in correspondence to \eqref{eq:chance-constraint}.
Without loss of generality, let us assume that the $i$th entry $\idx{\beta}{i}$ in the confidence vector $\beta$ corresponds to the $i$th constraint, 
so that
\begin{equation} \label{eq:DR-constraints}
 \DR{\Ufeas}(z) {=} \{ u \mid 
    \lriskc{\lrn}{\md}{\idx{\beta}{i}}{\DR{\alpha}_i}[g_i (x,u,\md,\mdnxt)] \leq 0,\forall i \in \natseq{1}{\nconst}\}.
\end{equation}
\begin{remark}
The parameters $\DR{\alpha}_i$ remain to be chosen in relation to the confidence levels $\beta$ and the original violation rates $\alpha_i$. In light of \Cref{prop:conditions-confidences}, $\DR{\alpha}_i= \tfrac{\alpha_i - \idx{\beta}{i}}{1-\idx{\beta}{i}}$ yields the least conservative choice. This choice is valid as long as it is ensured that $\idx{\beta}{i} \leq \alpha_{i}$.
\end{remark}

Let us denote the remaining element of the confidence vector 
$\beta$ corresponding to the cost by $\idx{\beta}{\costidx}$.
Using \eqref{eq:dr-expectation}, we can then express the \ac{DR} cost of a policy $\pol = (\pol_{k})_{k=0}^{\hor-1}$ as
\begin{multline} \label{eq:cost-function-risk}
    \DR{\cost}_\hor^{\pol}(z) \dfn \ell(x_0, u_0, \md_0) +
    \lrisk{\lrn_0}{\md_0}{\idx{\beta_{0}}{\costidx}}
    \big[
        \ell(x_1, u_1,\md_1)\\ +
         \lrisk{\lrn_1}{\md_1}{\idx{\beta_{1}}{\costidx}}
        \big[
            \dots + \lrisk{\lrn_{\hor-2}}{\md_{\hor-2}}{\idx{\beta_{\hor-2}}{\costidx}}\big[
             \ell(x_{\hor-1},u_{\hor-1}, \md_{\hor-1})\\
             + \lrisk{\lrn_\hor}{\md_\hor}{\idx{\beta_{\hor}}{\costidx}}[ \DR{\Vf}(x_\hor, \lrn_\hor, \beta_{\hor}, \md_\hor) ] \big]
            \dots 
        \big]
    \big],
\end{multline}
where $z_0=z$, $z_{k+1} =\fa(z_k, u_k, \md_{k+1})$ and $u_{k} = \pol_{k}(z_{k})$, for all $k \in \natseq{0}{\hor-1}$. 
In \Cref{sec:theory}, conditions on the terminal cost $\DR{\Vf}: \augset \rightarrow \barre_+: (x, \lrn, \beta, \md) \mapsto \Vf(x,\md) + \delta_{\DR{\Xf}}(x,\lrn, \beta)$ and its domain are provided in order to guarantee recursive feasibility and stability of the \ac{MPC} scheme defined by the following \ac{OCP}.

\begin{definition}[\acs{\RAOCP}] \label{def:raocp}
Given an augmented state $z \in \augset$, the optimal cost of the \ac{\RAOCP} is 
\begin{subequations} \label{eq:risk-averse-OCP}
\begin{align}
    \DR{\cost}_{\hor}(z) = \min_{\pol} \DR{\cost}_{\hor}^{\pol}(z)
\end{align}
subject to 
\begin{align} \label{eq:risk-averse-constraints}
    (x_0, \lrn_0, \beta_0, \md_0) &= z,\, \pol = (\pol_{k})_{k=0}^{\hor-1}, \\
    z_{k+1} &= \fa(z_k, \pol_{k}(z_k), \md_{k+1}),\\
    \pol_{k}(z_k) &\in \DR{\Ufeas}(z_k), \label{eq:risk-constraints-OCP}
    \; \forall \seq{\md}{0}{k} \in \W^{k},
\end{align}
for all $k \in \natseq{0}{\hor-1}$.
\end{subequations}
We denote by $\DR{\Pi}_{\hor}(z)$ the corresponding set of minimizers.
\end{definition}

We thus define the data-driven \ac{MPC} law analogously to the stochastic case as 
\begin{equation} \label{eq:MPC-law-DR}
    \DR{\law}_{\hor}(z) = \DR{\pol}_{0}^{\star}(z),
\end{equation}
where $(\DR{\pol}_{k}^{\star}(z))_{k=0}^{\hor-1} \in \DR{\Pi}_{\hor}(z)$.
At every time $t$, the data-driven MPC scheme thus consists of 
\begin{inlinelist}
    \item computing a control action $u_t = \DR{\law}_\hor(z_t)$ and applying it to the system \eqref{eq:system-dynamics}; 
    \item observing the outcome of $\md_{t+1} \in \W$ and the corresponding next state $x_{t+1}=f(x_t,u_t,\md_{t+1})$; and
    \item updating the learner state $\lrn_{t+1} = \learner(\lrn_t, \md_t,\md_{t+1})$ and the confidence levels $\beta_{t+1} = \confdyn(\beta_t)$, gradually decreasing the size of the ambiguity sets.
\end{inlinelist}

\begin{remark}[Scenario tree representations]
    Since $\W$ is a finite set, the possible realizations of $\seq{\md}{0}{\hor}$, given $z_0$ can be enumerated such that 
    the corresponding predicted states and controls 
    can be represented on a scenario tree\ifJournal, which is a directed graph with a single root node, corresponding to $\md_0$. Each node is connected to a number of child nodes, representing the possible realizations of the stochastic process at the next time step\fi \cite{pflug_MultistageStochasticOptimization_2014}. It therefore suffices to optimize over a finite number (equal to the number of nodes in the tree) of control actions instead of infinite-dimensional control laws.
    When represented in this manner, it is apparent that \eqref{eq:risk-averse-OCP} falls within the class of \emph{risk-averse, risk-constrained optimal control problems}, described in \cite{sopasakis_risk-averse_2019c}.
    In particular, the constraints \eqref{eq:risk-constraints-OCP} at stage $k$ can be represented in the framework of \cite{sopasakis_risk-averse_2019c} as \emph{nested risk constraints} which are compositions of a set of conditional risk mappings. In this case, the composition consists of $k-1$ $\max$ operators over values on the nodes in the first stages and a conditional risk mapping based on \eqref{eq:ambiguous-chance-constraint} at stage $k$. This is in line with the observations of \cite[Sec. 7.1]{kouvaritakis_ModelPredictiveControl_2016}. 
    Consequently, if the risk measures 
    employed in the definition of the \RAOCP~\eqref{eq:risk-averse-OCP} belong to the broad family of conic risk measures and the dynamics $f(\argdot, \argdot, w)$, $w \in \W$ are linear,
    then, using the reformulations in \cite{sopasakis_risk-averse_2019c}, 
    \eqref{eq:risk-averse-OCP} can be cast as a convex conic optimization problem.
    This is the case for many commonly used coherent risk measures, including the risk measure induced by the $\ell_1$-ambiguity set discussed in \Cref{ex:example-ambiguity-choice} (see \cite{schuurmans_LearningBasedRiskAverseModel_2020} for a numerical 
    case study). For nonlinear dynamics, the problem is no longer convex but can in practice still be solved effectively with standard NLP solvers.
\end{remark}\section{Theoretical analysis} \label{sec:theory}
\subsection{Dynamic programming}
    To facilitate theoretical analysis of the proposed \ac{MPC} scheme, we represent \eqref{eq:risk-averse-OCP} as a dynamic programming recursion, similarly to \cite{sopasakis_RiskaverseModelPredictive_2019}.
    We define the Bellman operator $\T$ as 
    \(
        \T(\hat{\cost})(z) \dfn \hspace{-1pt} \min_{u \in \DR{\Ufeas}(z)} \hspace{-3pt} \ell(x,u,\md) + \lrisk{\lrn}{\md}{\idx{\beta}{\costidx}} [\hat{\cost}(\fa(z,u,\mdnxt))],
    \)
    where $z = (x, \lrn, \beta, \md) \in \augset$ are fixed quantities and $\mdnxt \sim \row{\transmat}{\md}$.
    We denote by $\DParg (\DR{\cost})(z)$ the corresponding set of minimizers.
    The optimal cost $\hat{V}_N$ of \eqref{eq:risk-averse-OCP} is obtained through the iteration,
    \begin{equation} \label{eq:definition-DP}
        \begin{aligned}
        \hat{\cost}_k &= \T \hat{\cost}_{k-1}, \; \hat{\cost}_0 = \DR{\Vf}, \;k \in \natseq{1}{\hor}.
        \end{aligned}
    \end{equation}
    Similarly, $\DR{\augset}_{k} \dfn \dom \hat{\cost}_{k}$ is given recursively by 
    \begin{equation*}
        \DR{\augset}_k = \left\{z \, \middle|\, \exists u \in \DR{\Ufeas}(z) : 
            \fa(z,u,\mdnxt) \in \DR{\augset}_{k-1}, \,
            \forall \mdnxt \in \W \right\}.
    \end{equation*}

Now consider the stochastic closed-loop system
\begin{equation} \label{eq:closed-loop}
    \begin{aligned} 
    z_{t+1} = \fc(z_t, \md_{t+1}) \dfn \fa(z_t, \DR{\law}_\hor(z_t), \md_{t+1}),
    \end{aligned}
\end{equation}
where $\DR{\law}_{\hor}(z_t) \in \DParg(\hat{\cost}_{N-1})(z_t)$ is an optimal control law obtained by solving the data-driven \RAOCP{} of horizon $\hor$ in receding horizon.

\subsection{Constraint satisfaction and recursive feasibility}

In order to show existence of $\DR{\law}_\hor \in \DParg \DR{\cost}_{\hor-1}$ at every time step, \Cref{prop:recursive-feasibility} will
require that $\DR{\Xf}$ is a robust control invariant set. We define robust control invariance for the augmented control system under consideration as follows.

\begin{definition}[Robust invariance] \label{def:rob-inv}
    A set $\rinv \subseteq \augset$ is \iac{RCI} set for the system \eqref{eq:augmented-dynamics} if
    for all $z \in \rinv$, $\exists u \in \DR{\Ufeas}(z)$ such that $\fa(z,u,\mdnxt)  \in \rinv, \forall \mdnxt \in \W$.
    Similarly, $\rinv$ is \iac{RPI} set for the closed-loop system \eqref{eq:closed-loop} if for all $z \in \rinv$, $\fc(z,\mdnxt) \in \rinv,\, \forall \mdnxt \in \W$.
\end{definition}
Since $\DR{\Ufeas}$ consists of conditional risk constraints, \Cref{def:rob-inv} provides a \ac{DR} counterpart to the notion of \emph{stochastic} robust invariance in \cite{korda_StronglyFeasibleStochastic_2011}.
It is thus less conservative than the more classical notion of 
robust invariance for a set $\rinv_x$, obtained by imposing that $x \in \rinv_x \Rightarrow \exists u: g_i(x,u,\md,\mdnxt) \le 0$, for all $i \in \natseq{1}{\nconst}, \md,\mdnxt \in \W$. In fact, $\rinv_{x} \times \lrnset \times \I \times \W$ is covered by \Cref{def:rob-inv}.

\begin{proposition}[Recursive feasibility] \label{prop:recursive-feasibility}
    If $\DR{\Xf}$ is \iac{RCI} set for $\eqref{eq:augmented-dynamics}$, then \eqref{eq:risk-averse-OCP} is recursively feasible. That is,     
    feasibility of \ac{\RAOCP} \eqref{eq:risk-averse-OCP} for some $z \in \augset$, implies feasibility for $z^+ = \fc(z,\mdnxt)$, for all $\mdnxt \in \W, \hor \in \N_{>0}$.
\end{proposition}
\ifArxiv
    \begin{proof}
        We first show that if $\DR{\Xf}$ is \ac{RCI}, then so is $\DR{\augset}_\hor$.
        This is done by induction on the horizon $\hor$ of the \ac{OCP}. 
        
        \textbf{Base case ($\hor = 0$).} Trivial, since $\DR{\augset}_0 = \DR{\Xf}$.
        
        \textbf{Induction step ($\hor \Rightarrow \hor +1$).} Suppose that for some $\hor \in \N$, $\DR{\augset}_{\hor}$ is \ac{RCI} for \eqref{eq:augmented-dynamics}. Then, by definition of $\DR{\augset}_{\hor+1}$, there exists for each $z \in \DR{\augset}_{\hor+1}$, a nonempty set $\DR{\Ufeas}_{\hor}^\star(z) \subseteq \DR{\Ufeas}(z)$ such that for every $u \in \DR{\Ufeas}_{\hor}^{\star}(z)$ and for all $\mdnxt\in\W$, $z^+ \in \DR{\augset}_{\hor}$, where $z^+ = \fa(z,u,\mdnxt)$. Furthermore, the induction hypothesis ($\DR{\augset}_{\hor}$ is RCI), implies that there also exists a $u^+ \in \DR{\Ufeas}(z^+)$ such that $\fa(z^+, u^+, \mdnxt^+) \in \DR{\augset}_{\hor}(\mdnxt^+), \forall \mdnxt^+ \in \W$. Therefore, $z^+$ satisfies the conditions defining $\DR{\augset}_{\hor+1}$. In other words, $\DR{\augset}_{\hor+1}$ is RCI. 
        
        The claim follows from the fact that for any $\hor > 0$ and $z \in \DR{\augset}_{\hor}$, $u = \DR{\law}_{\hor}(z) \in \DParg(\DR{\cost}_{\hor-1})(z) \subseteq \DR{\Ufeas}^{\star}_{\hor-1}(z)$, as any other choice of $u$ would yield infinite cost in the definition of the Bellman operator.  
    \end{proof}
\fi

    \begin{cor}[Chance constraint satisfaction] \label{cor:constraint-satisfaction}
        If \Cref{prop:recursive-feasibility} holds, then by \Cref{prop:conditions-confidences}, the stochastic process $(z_t)_{t \in \N} = (x_t, \lrn_t, \beta_t, \md_t)_{t\in\N}$ satisfying dynamics \eqref{eq:closed-loop} satisfies the nominal chance constraints 
        \[ \prob[g_{i}(x_t, \DR{\law}_{\hor}(z_t), \md_{t+1})>0 \mid x_t, \md_t] < \alpha_i, \]
        \as, for all $i \in \natseq{1}{\nconst}, t \in \N$.
    \end{cor}
    We conclude this section by emphasizing that although 
    the MPC scheme guarantees closed-loop constraint satisfaction, it does so while being less conservative than a fully robust approach,
    which is recovered by taking 
    $\amb_{\beta}(s,\md) = \simplex_{\nModes}$ for all $s,\md,\beta$. It is apparent from \Cref{eq:ambiguous-chance-constraint,eq:DR-constraints}, that for all other choices of the ambiguity set, the set of feasible control actions will be larger (in the sense of set inclusion).

\subsection{Stability} \label{sec:stability}
We will now provide sufficient conditions on the control setup under which the origin is \ac{MSS} for \eqref{eq:closed-loop}, i.e., $\lim_{t \to \infty} \E[\nrm{x_{t}}^2] = 0$ for all $x_0$ in some specified compact set containing the origin.

Our main stability result, stated in \Cref{thm:MPC-stability},
hinges in large on the following \namecref{lem:DR-MSS}, which relates \emph{risk-square stability} {\cite[Lem. 5]{sopasakis_RiskaverseModelPredictive_2019}} of the origin for the autonomous system \eqref{eq:closed-loop} to stability in the mean-square sense, based on the statistical properties of the ambiguity sets.

\begin{lem}[\ac{DR} \ac{MSS} condition] \label{lem:DR-MSS}
    Suppose that \Cref{assum:confidences} holds and that there exists a nonnegative, proper function $\cost: \augset \rightarrow \barre_{+}$, such that 
    \begin{conditions*}
        \item \label{cond:invariance} $\dom V \subseteq \Xfeas \times \lrnset \times \I \times \W$ is \ac{RPI} for system \eqref{eq:closed-loop}, where $\Xfeas \subset \Re^{\ns}$ is a compact set containing the origin;
        \item \label{cond:lyap-decrease} $\lrisk{\lrn}{\md}{\idx{\beta}{\costidx}}[\cost(\fc(z,\mdnxt), \mdnxt)]-\cost(z) \leq - c \nrm{x}^2$, for some $c>0$,
        for all $z \in \dom \cost$;
        \item \label{cond:boundedness} $\cost$ is uniformly bounded on its domain, i.e., there exists a $\bar \cost \geq 0$, such that 
        $\cost(z) \leq \bar \cost$ for all $z \in \dom \cost$.
    \end{conditions*}
    Then, $\lim_{t \to \infty}\E[ \nrm{x_t}^2] = 0$ 
    for all $z_0 \in \dom \cost$, where $(z_t)_{t\in\N} = (x_t, \lrn_t, \beta_t, \md_t)_{t\in\N}$ is the stochastic process governed by dynamics \eqref{eq:closed-loop}.   
\end{lem}
\begin{proof}     
    See \Cref{proof:lem:DR-MSS}.
\end{proof}

\begin{thm}[\acs{MPC} stability] \label{thm:MPC-stability}
    Suppose that \Cref{assum:confidences} is satisfied, 
    and let $\bar{\augset} \dfn \bar{\Xfeas} \times \lrnset \times \I \times \W \subseteq \dom \DR{\cost}_{\hor}$ be an \ac{RPI} set for \eqref{eq:closed-loop}, where $\bar{\Xfeas} \subseteq \Re^{\ns}$ is a compact set containing the origin.
    Assume furthermore that the following statements hold:
    \begin{conditions*}
        \item \label{cond:TVfleqVf} $\T \DR{\Vf} \leq \DR{\Vf}$;
        \item \label{cond:stage-cost-bound} $c \nrm{x}^2 \leq \ell(x,u,\md)$
              for some $c > 0$
              for all
              $z = (x,\md,\beta,\lrn) \in \dom \DR{\cost}_\hor$ and all 
              $u \in \DR{\Ufeas}(z)$.
        \item \label{cond:local-bounded} $\DR{\cost}_\hor$ is locally bounded in $\bar \augset$.
    \end{conditions*}
    Then, the origin is \ac{MSS} for the \ac{MPC}-controlled system \eqref{eq:closed-loop}, for all initial states $z_0 \in \bar \augset$.
\end{thm}

\begin{proof}
    The proof is along the lines of that of \cite[thm. 6]{sopasakis_RiskaverseModelPredictive_2019} and shows that $\hat{\cost}_\hor$ satisfies the conditions of \Cref{lem:DR-MSS}. Details are in the \Cref{proof:thm:MPC-stability}.
\end{proof}

\section{Conclusions}
We presented a distributionally robust \ac{MPC} strategy for Markovian switching systems with unknown transition probabilities subject to general chance constraints. 
using data-driven, high-confidence ambiguity sets, we derive a \ac{DR} counterpart to a nominal stochastic MPC scheme.
We show that the resulting scheme provides a priori guarantees on closed-loop constraint satisfaction and mean-square stability of the true system, without requiring explicit knowledge of the transition probabilities.

\bibliographystyle{ieeetr}
\bibliography{references}
\vspace{-0.5em}
\begin{appendix}
  \vspace{-0.5em}

\begin{appendixproof}{lem:DR-MSS}
    Let $(z_t)_{t\in\N}=(x_t, \lrn_t, \beta_t, \md_t)_{t\in\N}$ denote 
    the stochastic process satisfying dynamics \eqref{eq:closed-loop}, for some initial state $z_0 \in \dom \cost$. For ease of notation, let us define $\cost_t \dfn \cost(z_t), t \in \N$. 
    Due to nonnegativity of $V$, 
    \[
        \begin{aligned}
            \E\left[\tsum_{t=0}^{k-1} c \nrm{x_t}^2\right]
            &\leq \E\left[V_k + \tsum_{t=0}^{k-1} c \nrm{x_t}^2\right]\\
            & =\E\left[V_k - V_0 + \tsum_{t=0}^{k-1} c \nrm{x_t}^2\right] + V_0,
        \end{aligned}
    \]
    where the second equality follows from the fact that $\cost_0$ is deterministic. By linearity of the expectation, we can in turn write
    \[
    \begin{aligned}
        \E\big[V_k{-}V_0{+}c \tsum_{t=0}^{k-1} \nrm{x_t}^2 &\big] 
                = \E\left[\tsum_{t=0}^{k-1} V_{t+1}{-}V_{t} {+} c \nrm{x_t}^2 \right]\\  
                &= \tsum_{t=0}^{k-1} \E \left[ V_{t+1} {-} V_t{+} c \nrm{x_t}^2 \right].
    \end{aligned}
    \]
    Therefore,
    \begin{equation} \label{eq:step-proof}
        \begin{aligned}
            \E\big[c \tsum_{t=0}^{k-1} \nrm{x_t}^2\big]&{-}V_0
            \leq
            \tsum_{t=0}^{k-1} \E \left[V_{t+1}{-}V_t \right]{+}c \E \left[ \nrm{x_t}^2 \right].
        \end{aligned}
    \end{equation}
    Recall that $\idx{\beta}{\costidx}$ denotes the coordinate of $\beta$ corresponding to the risk measures in the cost function \eqref{eq:cost-function-risk}. Defining the event
    \(
        E_t \dfn \{
                \omega \in \Omega \mid
                    \row{\transmat}{\md_t(\omega)} \in
                    \amb_{\idx{\beta_t}{\costidx}}(\lrn_t(\omega), \md_t(\omega))
            \},
    \)
    and its complement $\lnot E_t = \Omega \setminus E_t$, we can use the law of total expectation to write 
    \begin{multline*}
        \E\left[ V_{t+1} - V_t \right]
        = \E\left[ V_{t+1} - V_{t} \mid E_t \right] {\prob[E_t]} \\
        +\E\left[ V_{t+1} - V_{t} \mid \lnot E_t \right] {\prob[\lnot E_t]}.
    \end{multline*}
    By condition \eqref{eq:high-confidence}, $\prob[\lnot E_t] < \idx{\beta_{t}}{\costidx}$. From condition \ref{cond:invariance}, it follows that $z_t \in \dom \cost$, $\forall t \in \natseq{0}{k}$ and that there exists a $\dV \geq 0$ such that $V(z) \leq \dV$, for all $z \in \dom V$. Therefore, $\E[V_{t+1} - V_t \mid \lnot \E_t] \leq \dV$.
    Finally, by condition \ref{cond:lyap-decrease}, 
    \(
        \E\left[ V_{t+1}- V_{t} \mid E_t \right] 
        \leq \E[-c \nrm{x_t}^2].
    \) Thus,
    \extrastep{
        This follows from explicitly writing out the expectation 
        \[ 
            \begin{aligned}
                \E[V_{t+1} - V_t] &= \E[\smashoverbracket{\E[V_{t+1} \mid V_t] }{\leq V_t - c \nrm{x}^2} ] - \E[V_t] \\
                &\leq \E[V_t - c \nrm{x}^2] - \E[V_t] \\
                &= - c \nrm{x}^2 
            \end{aligned}
        \]
    }
    \begin{align*}
        \E\left[ V_{t+1} - V_t \right] 
        &\leq \E\left[-c \nrm{x_t}^2 \mid E_t\right] \prob[E_t] + \dV \idx{\beta_{t}}{\costidx}.
    \end{align*}
    This allows us to simplify expression \eqref{eq:step-proof} as
    \begin{align*}
        &\E\left[c \tsum_{t=0}^{k-1} \nrm{x_t}^2\right] -V_0\\
        &\leq \tsum_{t=0}^{k-1}-c \E\left[\nrm{x_t}^2 \mid E_t\right] \prob[E_t] + \dV \idx{\beta_{t}}{\costidx} 
        + c \E \left[ \nrm{x_t}^2 \right]\\  
        &\leq \tsum_{t=0}^{k-1} -c \E\left[\nrm{x_t}^2 \mid E_t\right] \prob[E_t] + \dV \idx{\beta_{t}}{\costidx}     
        + \\ 
        &\qquad c \E \left[ \nrm{x_t}^2 \mid E_t \right] \prob[E_t] 
         +c \E \left[ \nrm{x_t}^2 \mid \lnot E_t \right] \prob[\lnot E_t]\\  
        &= \tsum_{t=0}^{k-1} \dV \idx{\beta_{t}}{\costidx} +c \E \left[ \nrm{x_t}^2 \mid \lnot E_t \right] \prob[\lnot E_t]\\
        &\leq \tsum_{t=0}^{k-1} \idx{\beta_{t}}{\costidx} (\dV +  c \E\left[\nrm{x_t}^2 \mid \lnot E_t \right] ).
    \end{align*}
    Since $x_t \in \Xfeas, t \in \N$, and $\Xfeas$ is a compact set containing the origin, there exists an 
    $r \geq 0$ such that $\nrm{x}^{2} \leq r$. Therefore,   
    \begin{align*}
        \E\left[\tsum_{t=0}^{k-1} \nrm{x_t}^2\right]
        &\leq \tfrac{V_0}{c} + \left(\tfrac{\dV}{c} + r \right) \tsum_{t=0}^{k-1} \idx{\beta_{t}}{\costidx},
    \end{align*}
    which remains finite as $k \to \infty$, since $(\idx{\beta_{t}}{\costidx})_{t \in \N}$ is summable.
    Thus, necessarily, $\lim_{t\to \infty} \E[\nrm{x_t}^2] = 0$.
\end{appendixproof}
\begin{appendixproof}{thm:MPC-stability}
    First, note that using the monotonicity of coherent risk measures \cite[Sec. 6.3, (R2)]{shapiro2009lectures}, a straightforward inductive argument allows us to show that 
    under \Cref{cond:TVfleqVf},
    \begin{equation} \label{eq:monotonicity-T}
        \T \DR{V}_\hor \leq \DR{V}_\hor, \quad \forall \hor \in \N.
    \end{equation}
    Since $\bar{\mathcal{Z}} \subseteq \dom \DR{\cost}_\hor$, recall that by definition \eqref{eq:definition-DP}, we have for any $z \in \bar{\mathcal{Z}}$ that
    \begin{multline*}
        \hat{\cost}_N(z) = \ell(x,\DR{\law}_\hor(z), \md) 
        + \lrisk{\md}{\lrn}{\beta}\big[
            \hat{\cost}_{N-1}
            \big(
                \fc(z,\mdnxt)
            \big)
        \big].
    \end{multline*} 
    Therefore, we may write   
    \begin{align*}
        &\lrisk{\md}{\lrn}{\beta}\left[
            \hat{\cost}_\hor(\fc(z,\mdnxt))
        \right] - 
        \hat{\cost}_\hor(z)  \\ 
        &=\lrisk{\md}{\lrn}{\beta}\left[
            \hat{\cost}_\hor(\fc(z,\mdnxt))
        \right] 
        - \ell(x,\DR{\law}_\hor(z), \md) 
        \\ &\; - \lrisk{\md}{\lrn}{\beta}
        \big[
        \hat{\cost}_{N-1}
        \big(
            \fc(z,\mdnxt)
        \big)
        \big]
        \leq - \ell (x, \DR{\law}_\hor(z), \md) \leq - c \nrm{x}^2,
    \end{align*}
    where the first inequality follows by \eqref{eq:monotonicity-T} and monotonicity of coherent risk measures. The second inequality follows from \Cref{cond:stage-cost-bound}.
    Therefore, $\cost: z \to \DR{\cost}_\hor(z) + \delta_{\bar{\mathcal{Z}}}(z)$ satisfies the conditions of \Cref{lem:DR-MSS} and the assertion follows. 
\end{appendixproof}

  \end{appendix}
\end{document}